\documentclass[10pt]{amsart}
\usepackage{amscd}
\usepackage{amssymb}
\usepackage{amsmath}
\usepackage{amsfonts}
\usepackage{enumerate}
\usepackage[all,cmtip]{xy}
\usepackage{tikz-cd}
\usepackage{hyperref}

\newtheorem{thm}{Theorem}[section]
\newtheorem{prop}[thm]{Proposition}
\newtheorem{lem}[thm]{Lemma}
\newtheorem{cor}[thm]{Corollary}

\theoremstyle{definition}
\newtheorem{Def}[thm]{Definition}
\newtheorem{exmp}[thm]{Example}

\newtheorem{rem}[thm]{Remark}

\newcommand{\threerightarrows}{%
\mathrel{\vcenter{\hbox{$\rightrightarrows$}\nointerlineskip\hbox{$\rightarrow$}}}
}
\newcommand{\fourrightarrows}{%
\mathrel{\vcenter{\hbox{$\threerightarrows$}\nointerlineskip\hbox{$\rightarrow$}}}
}

\numberwithin{equation}{section}

\title{on topological coHochschild homology and  cotensor products}

\author[J. Zha]{Jiaxi Zha}
\address{Department of Mathematics, Nankai University, No.94 Weijin Road, Tianjin 300071, P. R. China}
\email{1093913699@qq.com}

\begin{document}

\begin{abstract}
	In this work, we first study the cotensor product of comodules in the $\infty$-category $\mathrm{Mod}_R$ for a connective $\mathbb{E}_{\infty}$-ring spectrum $R$. We then apply these results to analyze higher coalgebra structures of topological coHochschild homology (coTHH) and establish its Morita-Takeuchi invariance, which are precisely dual to the corresponding properties of topological Hochschild homology.
\end{abstract}

\maketitle

\section{Introduction}
Dual to the notion of Hochschild homology, coHochschild homology was first introduced by Doi in \cite{Doi81} as an invariant of coalgebras. Building upon this foundation, Hess-Shipley introduced topological coHochschild homology for coalgebras in the category of spectra in \cite{HS21}, which extended Doi's definition to the stable homotopy theory. Furthermore, they proved that for any  EMSS-good space $X$, there is an equivalence
\[\mathrm{coTHH}(\Sigma_+^{\infty} X)\simeq \Sigma_+^{\infty}LX,\]
 extending Malkiewich's result for simply connective spaces in \cite{Mal17}. Based on this equivalence, new computational tools for the homology of the free loop space of $X$ have been developed. We refer to \cite{BGS22} and \cite{Kla22} for concrete applications. An analogous formula for topological Hochschild homology can be found in \cite{Go85,NS18}, and its categorical proof through higher categorical traces is given in \cite {CCRY22, HSS17}. While we intend to apply the methods of \cite{CCRY22} to study properties of coTHH, this approach lies beyond the scope of the current work and will be addressed in future research. Another motivation comes from the coalgebra itself containing substantial information regarding topological spaces, with its invariant providing measurable data. For example, using the category of comodules over the coalgebra $\Sigma_+^{\infty}X$, Hess and Shipley reconstructed the Waldhausen K-theory of a topological space $X$  \cite{HS16}.  

Most of aforementioned work on coTHH operates within specific model categories of spectra, while an interesting result \cite{PS18} of P\'{e}roux and Shipley shows that all coalgebra spectra are automatically cocommutative in certain model categories of spectra. The proof depends on the existence of strict monoidal structures of these model categories, a condition that is not satisfiable in the $\infty$-category setting. This limitation motivates our adoption of $\infty$-categorical language. Specifically, we study coalgebras and comodules in the $\infty$-category $\mathrm{Mod}_R$ where $R$ is an $\mathbb{E}_{\infty}$-ring spectrum. In fact, a recent work \cite{BP23} of Bayindir and P\'{e}roux established a general definition of coTHH in arbitrary symmetric monoidal $\infty$-categories and investigated the relationship between topological Hochschild homology and coTHH via Spanier-Whitehead duality. 

Consequently, a fundamental question arises: Does coTHH admit dualized analogues of properties satisfied by topological Hochschild homology? The primary objective of this work is to study the relationship between coTHH and the cotensor product, which constitutes the dual analogue of the fundamental connection between topological Hochschild homology and the tensor product. A fundamental theorem to compute topological Hochschild homology states that
    \[\mathrm{THH}(M;A)\simeq M\otimes_{A\otimes A^{\mathrm{op}}}A.\]
    Our first main result is the dual analogue for coTHH. While this should be familiar to experts, we include it in the paper for completeness. Further similar computational formulas will be established in Section \ref{sec: app}.
 
\begin{thm}\label{th1}
	Let $\mathcal{C}$ be a symmetric monoidal $\infty$-category which admits small limits, and let $C$ be a coalgebra in $\mathcal{C}$ and $M$ be a $C$-bicomodule. Then there is an equivalence
\begin{equation*}
	\mathrm{coTHH}^{\mathcal{C}}(M;C) \simeq M \ \square_{C\otimes C^{\mathrm{op}}} \ C.
\end{equation*}	
\end{thm} 

\begin{rem}
	Let $\mathbb{K}$ be a commutative ring with global dimension zero. In the model categorical setting, \cite{KP25} established an analogous formula for a simply connective $\mathbb{K}$-coalgebra and a fibrant connective $C$-bicomodule $M$:
	\[\mathrm{coHH}(M;C)\simeq \Omega(M,C\otimes C^{\mathrm{op}},C).\]
\end{rem}

Next, we turn to studying coTHH of $R$-coalgebras, where $R$ is a connective $\mathbb{E}_{\infty}$-ring spectrum.  In general, many dual properties fail for coTHH because the tensor product does not necessarily preserve limits separately in each variable in the $\infty$-category of $R$-modules. To address this, we introduce some connectivity conditions (detailed in Section 3), under which the cotensor product have the following property:

\begin{thm}\label{th: sym}
	The cotensor product restricts to a symmetric monoidal functor:
	\[\mathrm{BiCoMod}(\mathrm{Mod}_R^{\mathrm{cn}})\times_{\mathrm{CoAlg}(\mathrm{Mod}_R)_{R/}^{\geq 2}}\mathrm{BiCoMod}(\mathrm{Mod}_R^{\mathrm{cn}}) \to \mathrm{BiCoMod}(\mathrm{Mod}_R),\]
	where $\mathrm{CoAlg}(\mathrm{Mod}_R)_{R/}^{\geq 2}$ is the $\infty$-category of coaugmented coalgebras with $2$-connective coaugmentation coideals.
\end{thm}

\begin{rem}\label{bounded}
	In fact, it suffices to require the comodules to be bounded below, since we can reduce the proof to the connective case after finite suspensions.
\end{rem}

As an application of this theorem, we establish the following result concerning higher coalgebra structures for coTHH, which is the dual analogue of the well-known structures for topological Hochschild homology.

\begin{thm}\label{thm:coalg}
    The functor $\mathrm{coTHH}^R$ restricts to a functor
	\[\mathrm{coTHH}^R: \mathrm{CoAlg}_{\mathbb{E}_k}(\mathrm{Mod}_R)_{R/}^{\geq 2}\to \mathrm{CoAlg}_{\mathbb{E}_{k-1}}(\mathrm{Mod}_R),\]
	for $1\leq k\leq \infty$. Moreover, if $C\in \mathrm{CoCAlg}(\mathrm{Mod}_R)_{R/}^{\geq 2}$, then the map $\mathrm{coTHH}^R(C) \to C$ is terminal among all maps from an cocommutative R-coalgebra equipped with an $S^1$ action to $C$ through cocommutative maps.
\end{thm}

\subsection{Orgnization}
In section \ref{sec: def}, we recall definitions of coTHH and cotensor products, provide several examples, and give the proof of Theorem \ref{th1}. After this, we establish some formal properties of cotensor products in the $\infty$-category $\mathrm{Mod}_R$ for a connective $\mathbb{E}_{\infty}$-ring spectrum $R$ in section \ref{sec: prop of cotensor} and give a proof of Theorem \ref{th: sym}. Finally, in section \ref{sec: app}, we apply these results to prove Theorem \ref{thm:coalg}, study higher coalgebra structures on coTHH and Morita-Takeuchi invariance of coTHH, under suitable connectivity assumptions.

\subsection{Conventions}
We work with $\infty$-categories throughout, following \cite{HTT,HA}. All $\infty$-categories are assumed to be complete, and the limit of a cosimplicial object $X^{\bullet}$ of $\mathcal{C}$ is denoted by $\mathrm{Tot}(X^{\bullet})$, while the limit restricted to \(\Delta|_{\leq n}\) is denoted by $\text{Tot}^n(X^{\bullet})$. For an ordinary category $\mathcal{C}$, we regard it as an $\infty$-category through the nerve functor, and by abuse of notation we still denote it by $\mathcal{C}$. Finally, a spectrum $X\in \mathrm{Sp}$ is called $n$-connective if $\pi_i(X)=0$ for $i\leq n$. It is said to be connective, if it is $0$-connective.

\section{coHochschild homology and cotensor product}\label{sec: def}
In this section, we firstly review the definition of $\mathcal{O}$-coalgebras in a general symmetric monoidal $\infty$-category, where $\mathcal{O}$ is an $\infty$-operad. Then we present the definitions of (relative) coTHH and the cotensor product, and prove Theorem \ref{th1}. 

Given a symmetric monoidal $\infty$-category $\mathcal{C}$, let us first recall the construction of the symmetric monoidal structure on $\mathcal{C}^{\mathrm{op}}$ in \cite{BGN14}. Let $p:\mathcal{C}^{\otimes} \to \mathrm{N}(\mathrm{Fin}_*)$ be the coCartesian fibration associated to the symmetric monoidal structure of $\mathcal{C}$, which is classified by the functor $F:\mathrm{N}(\mathrm{Fin}_*) \to \widehat{\mathrm{Cat}}$. The same functor $F$ also classifies a Cartesian fibration $p^{\vee}:(\mathcal{C}^{\otimes})^{\vee}\to \mathrm{N}(\mathrm{Fin}_*)^{\mathrm{op}}$. Hence, we obtain another coCartesian fibration $(p^{\vee})^{\mathrm{op}}:((\mathcal{C}^{\otimes})^{\vee})^{\mathrm{op}}\to \mathrm{N}(\mathrm{Fin}_*)$ whose fiber over $\langle 1 \rangle\in \mathrm{N}(\mathrm{Fin}_*)$ is $\mathcal{C}^{\mathrm{op}}$. One can check that $(p^{\vee})^{\mathrm{op}}$ is a coCartesian fibration of $\infty$-operad.

\begin{Def}\cite[Definition 2.1]{Pe22}
	Let $\mathcal{C}$ be a symmetric monoidal $\infty$-category and $\mathcal{O}^{\otimes}$ be an $\infty$-operad. The $\infty$-category of $\mathcal{O}$-coalgebra in $\mathcal{C}$ is defined as 
	\[\mathrm{CoAlg}_{\mathcal{O}}(\mathcal{C}):= \mathrm{Alg}_{\mathcal{O}}(\mathcal{C}^{\mathrm{op}})^{\mathrm{op}}.\]
	 Given a coalgebra $C$ in $\mathcal{C}$, the $\infty$-category of right $C$-comodules in $\mathcal{C}$ is defined as 
	\[\mathrm{CoMod}_C(\mathcal{C}):= \mathrm{Mod}_C(\mathcal{C}^{\mathrm{op}})^{\mathrm{op}}=(\mathrm{Alg}_{\mathcal{RM}}(\mathcal{C}^{\mathrm{op}})\times_{\mathrm{Alg}(\mathcal{C}^{\mathrm{op}})} C)^{\mathrm{op}}.\]
	The $\infty$-categories of left $C$-comodules, $C$-$D$-bicomodules and bicomodule objects are defined similarly.
\end{Def}

For notational convenience, we write $\mathrm{CoAlg}(\mathcal{C})$ and $\mathrm{CoCAlg}(\mathcal{C})$ for the $\infty$-categories of $\mathbb{E}_1$-coalgebras and $\mathbb{E}_{\infty}$-coalgebras in $\mathcal{C}$, which we shall call coalgebras and cocommutative coalgebras respectively. For a coalgebra $C$, we denote its opposite coalgebra by $C^{\mathrm{op}}$ (see \cite[Remark 4.1.1.7]{HA}). By \cite[Proposition 3.2.4.7]{HA}, $\mathrm{CoCAlg}(\mathcal{C})$ admits a Cartesian symmetric monoidal structure given by pointwise tensor product. Moreover, \cite[Construction 3.2.4.1]{HA} endows $\mathrm{BiCoMod}$, the $\infty$-category of bicomodule objects, with a symmetric monoidal sturcture. To avoid confusion, we will use the symbol $\boxtimes$ to denote the tensor product on $\mathrm{BiCoMod}$.

In what follows, we formally define the topological coHochschild homology. For a comprehensive treatment of this construction, the reader is referred to \cite{BP23} and \cite{KP25}.
\begin{Def}
	Let $\mathcal{C}$ be a symmetric monoidal $\infty$-category, $C$ be a coalgebra in $\mathcal{C}$ and $M$ be a $C$-bicomodule. 
	\begin{enumerate}
		\item The topological coHochschild homology of $C$ is the topological Hochschild homology of $C$ considered as an algebra of $\mathcal{C}^{\mathrm{op}}$. We denote it by $\mathrm{coTHH}^{\mathcal{C}}(C)$.
		\item Similarly, the topological coHochschild homology of $C$ with coefficient $M$ is the the topological Hochschild homology of $C$ as an algebra of $\mathcal{C}^{\mathrm{op}}$ with coefficient $M$ as a $C$-bimodule in $\mathcal{C}^{\mathrm{op}}$. We denote it by  $\mathrm{coTHH}^{\mathcal{C}}(M;C)$.
	\end{enumerate}
	The existence of these objects is guaranteed by the completeness of $\mathcal{C}$.
\end{Def}

Unwinding the definitions, we see that $\mathrm{coTHH}^{\mathcal{C}}(C)$ is the totalization of the following cyclic cobar construction:
\begin{equation*}
	C \rightrightarrows C\otimes C \threerightarrows C\otimes C \otimes C \fourrightarrows \cdots .
\end{equation*}
Note that, since $\mathrm{THH}^{\mathcal{C}^{\mathrm{op}}}$ is a functor from $\mathrm{Alg}(\mathcal{C}^{\mathrm{op}})$ to $\mathcal{C}^{\mathrm{op}}$, by taking opposites we obtain 
\[\mathrm{coTHH}^{\mathcal{C}}:\mathrm{CoAlg}(\mathcal{C})\to \mathcal{C}\]
as a functor. Similarly, $\mathrm{coTHH}^{\mathcal{C}}(-;C)$ is a functor from $\mathrm{BiCoMod}_C(\mathcal{C})$ to $\mathcal{C}$.

\begin{exmp}\label{example}
	Let $\mathcal{C}$ be a symmetric monoidal $\infty$-category such that the tensor product preserves totalizations separately in each variable. Given a cocommutative coalgebra $C\in \mathcal{C}$, we can consider it as an object in $\mathrm{CAlg}(\mathcal{C}^{\mathrm{op}})$. Moreover, it is well known that
		$\mathrm{THH}^{\mathcal{C}^{\mathrm{op}}}(C) \simeq C\otimes S^1\in \mathrm{CAlg}(\mathcal{C}^{\mathrm{op}}).$
		Hence, we deduce that 
		\[\mathrm{coTHH}^{\mathcal{C}}(C)\simeq C^{S^1}\in \mathrm{CoCAlg}(\mathcal{C}),\]
		where $C^{S^1}$ is the cotensor of $C$ with $S^1$.
\end{exmp}

Given a coalgebra $C$ in a symmetric monoidal $\infty$-category $\mathcal{C}$ and a right $C$-comodule $M$, we may regard $M$ as a $\mathbf{1}$-$C$-bicomodule by \cite[Corollary 4.3.2.8]{HA}, where $\mathbf{1}$ is the unit of $\mathcal{C}$. Similarly, a left $C$-comodule can be viewed as a $C$-$\mathbf{1}$-bicomodule. As established in \cite[Construction 4.4.2.7]{HA}, there exists a canonical two-sided bar construction on $M$ and $N$ in $\mathcal{C}^{\mathrm{op}}.$ Consequently, we obtain a cosimplicial object $\Omega_{\mathcal{C}}^{\bullet}(M,C,N)$ in $\mathcal{C}$, called the two-sided cobar construction on $M$ and $N$.

\begin{Def}
	Let $\mathcal{C}$ be a symmetric monoidal $\infty$-category, $C$ be a coalgebra in $\mathcal{C}$, and let $M$ be a right $C$-comodule and $N$ be a left $C$-comodule. The cotensor product of $M$ and $N$ is the totalization of the cosimplicial object $\Omega_{\mathcal{C}}^{\bullet}(M,C,N)$  in $\mathcal{C}$, denoted by $M\square_C N$.
\end{Def}

In other words, $M\square_C N$ is the totalization of the following cobar construction:
\[ M\otimes N\rightrightarrows M\otimes C\otimes N \threerightarrows M\otimes C\otimes C\otimes N \fourrightarrows \cdots.\]
It is clear that the definition is natural in both $C$, $M$ and $N$, hence we obtain a functor
\[\square:\mathrm{RCoMod}(\mathcal{C}) \times_{\mathrm{CoAlg}(\mathcal{C})}\mathrm{LCoMod}(\mathcal{C}) \to \mathcal{C}.\]

\begin{rem}
	Unlike \cite[Definition 4.4.2.10]{HA}, since we do not assume the tensor product preserve $\Delta$-indexed limits separately in each variable, the functor $\square$ does not necessarily lift to a functor from ${_A\mathrm{CoMod}_C}\times {_C\mathrm{CoMod}}_B$ to $_A\mathrm{CoMod}\mathcal{C}_B$. However, such a lifting may be constructed under certain connectivity assumptions, see Proposition \ref{prop:suqare in LcMod}.
\end{rem}

Nonetheless, given a coalgebra $C\in\mathcal{C}$, the dual of \cite[Example, 4.7.2.5]{HA} shows that $\mathrm{CoMod}_C(\mathcal{C})$ is comonoidic  over $\mathcal{C}$, hence the equivalence $M\simeq M\square_C C$ in $\mathcal{C}$ can be lifted to $\mathrm{CoMod}_C(\mathcal{C})$. Furthermore, we have the following stronger version:

\begin{lem}\label{lem: sq C}
	Given coalgebras $C,D$ in $\mathcal{C}$, the functor $-\square_C C:\mathrm{CoMod}_C\to \mathcal{C}$ restricts to a functor from ${_D\mathrm{CoMod}_C}$ to ${_D\mathrm{CoMod}_C}$. Similarly, the functor $D\square_D -:{_D\mathrm{CoMod}}\to \mathcal{C}$ restricts to a functor from ${_D\mathrm{CoMod}_C}$ to ${_D\mathrm{CoMod}_C}$. Furthermore, both functors are identity functors.
\end{lem}
\begin{proof}
	Given $M\in {_D\mathrm{CoMod}_C}$, $M\square_C C$ is the totalization of the following cobar construction in the $\infty$-category $\mathcal{C}$:
\[\Omega_{\mathcal{C}}^{\bullet}(M,C,C): M\otimes C \rightrightarrows M\otimes C\otimes C \threerightarrows M\otimes C\otimes C\otimes C \fourrightarrows \cdots.\]
This cosimplicial object splits with the coaugmentation given by $M\to M\otimes C$, the coaction map of $M$. In particular, for any object $X\in \mathcal{C}$, the cosimplicial object $X\otimes \Omega_{\mathcal{C}}^{\bullet}(M,C,C)$ also splits with the totalization 
    \[\mathrm{Tot}(X\otimes \Omega_{\mathcal{C}}^{\bullet}(M,C,C))\simeq X\otimes M\simeq X\otimes \mathrm{Tot}(\Omega_{\mathcal{C}}^{\bullet}(M,C,C)).\] 
    Now the results follows from \cite[Corollary, 4.2.3.5]{HA} and the equivalence of $\infty$-categories
    \[{_D\mathrm{CoMod}_C}(\mathcal{C})\simeq {_{D\otimes C^{\mathrm{op}}}} \mathrm{CoMod}(\mathcal{C}),\]
    see \cite[\S 4.6.3]{HA}. Furthermore, as noted above, the limit of the splitting cosimplicial object $\Omega_{\mathcal{C}}^{\bullet}(M,C,C)$ is $M$. Hence, $-\square_C$ is the identity functor.
\end{proof}

In particular, taking $D=C$ in the above Lemma and regarding $C$ as a $C$-bicomodule, we see that $\mathrm{Tot}_{\mathcal{C}}(\Omega_R^{\bullet}(C,C,C))$ admits a $C$-bicomodule structure, and there is an equivalence 
\begin{equation}\label{bicomod}
	C\square_C C\stackrel{\text{def}}{=}\mathrm{Tot}_{\mathcal{C}}(\Omega_R^{\bullet}(C,C,C))\simeq C\in\mathrm{\mathrm{BiCoMod}_C}(\mathcal{C}).
\end{equation}
Here, we write the totalization as $\mathrm{Tot}_{\mathcal{C}}(X^{\bullet})$ to emphasize that the limit is taken in $\mathcal{C}$

The proof of Theorem \ref{th1} is now straightforward.

\begin{thm}\label{cot and cothh}
Let $C$ be a coalgebra in $\mathcal{C}$, and let $M$ be a $C$-bicomodule, then there is an equivalence
\begin{equation}\label{coTHH and sq}
	M \square_{C^e} C\simeq \mathrm{coTHH}^{\mathcal{C}}(M;C),
\end{equation}
    where $C^e=C\otimes C^{\mathrm{op}}$ is the enveloping coalgebra of $C$.
\end{thm}
\begin{proof}
	It follows from Equation \ref{bicomod} that
	\begin{equation*}
		\begin{split}
			M\ \square_{C^e}C & \stackrel{\text{def}}{=} \lim_{[i] \in \Delta} (M\otimes (C^e)^{\otimes i}\otimes C)\\
			& \simeq \lim_{[i] \in \Delta}[M\otimes (C^e)^{\otimes i}\otimes \lim_{[j] \in \Delta}(C\otimes C^{\otimes j}\otimes C).]
			\end{split}
	\end{equation*}		
	Moreover, by Equation \ref{bicomod} and the fact that $\Omega^{\bullet}_{\mathcal{C}}(C,C,C)$ splits, the above is equivalent to
	\begin{equation*}
		\begin{split}
		    & \lim_{[i] \in \Delta} \lim_{[j] \in \Delta}[M\otimes (C^e)^{\otimes i}\otimes (C\otimes C^{\otimes j}\otimes C)]\\		
			& \simeq \lim_{[j] \in \Delta} \lim_{[i] \in \Delta}(M\otimes (C^e)^{\otimes i}\otimes C^e \otimes C^{\otimes j})\\
			& \simeq \lim_{[j] \in \Delta} [\lim_{[i] \in \Delta}(M\otimes (C^e)^{\otimes i}\otimes C^e) \otimes C^{\otimes j}]\\
			& \simeq \lim_{[j] \in \Delta} (M \otimes C^{\otimes j}) \stackrel{\text{def}}{=} \mathrm{coTHH}^R(M,C).
		\end{split}
	\end{equation*}
	 Similarly, the splitting of $\Omega^{\bullet}_{\mathcal{C}}(M,C^e,C^e)$ yields the third equivalence. The other two equivalences are immediately. 
\end{proof}

\begin{exmp}
		Let $\mathcal{C}$ be a presentable symmetric monoidal $\infty$-category, if the underlying objects of an algebra $A$ and an $A$-bimodule $M$ are dualizable, then $A^{\vee}$ is an coalgebra and $M^{\vee}$ is a $A^{\vee}$-bicomodule, since $A^{\vee}\simeq \mathrm{hom}_{\mathcal{C}}(A,\mathbf{1})$, $M^{\vee}\simeq \mathrm{hom}_{\mathcal{C}}(M,\mathbf{1})$ and $\mathrm{hom}_{\mathcal{C}}(-,\mathbf{1})$ is symmetric monoidal when restricted to dualizable objects. Also note that $(M\otimes A^{\otimes_i})^{\vee}\simeq M^{\vee}\otimes (A^{\vee})^{\otimes i}$. Hence, by Equation \ref{coTHH and sq}, we obtain that 
	\[\mathrm{coTHH}_{\mathcal{C}}(M^{\vee};A^{\vee})\simeq M^{\vee}\square_{(A^{\vee})^e} A^{\vee}\simeq (M \otimes_{A^e} A)^{\vee} \simeq \mathrm{THH}^{\mathcal{C}}(M;A)^{\vee}\]
    since $\mathrm{hom}_{\mathcal{C}}(-,\mathbf{1})$ sends colimits to limits. A more comprehensive treatment of this topic can be found in \cite{BP23}.
\end{exmp}

\begin{exmp}[\cite{HS21,Mal17}]
	Consider the $\infty$-category of spaces, denoted by $\mathcal{S}$, with Cartesian symmetric monoidal structure, any object $X\in \mathcal{S}$ is a cocommutative coalgebra with the structure map given by the diagonal map. Since $\Sigma_+^{\infty}$ is symmetric monoidal, given an EMSS-good space $X$, we have
    \begin{equation*}
	\begin{split}
	 \mathrm{coTHH}^{\mathbb{S}}(\Sigma_+^{\infty} X) & \simeq \Sigma_+^{\infty} X\square_{\Sigma_+^{\infty} (X\times X)}\Sigma_+^{\infty} X\\
		& \simeq \Sigma_+^{\infty}(X \square_{(X\times X)} X)\\
		& \simeq \Sigma_+^{\infty} X^{S^1},
	\end{split}
	\end{equation*}
	where the second equivalence follows from \cite[Lemma 3.17]{HS21} and the third equivalence follows from Example \ref{example}. Furthermore, there are equivalences
	\[X^{S^1} \simeq \mathrm{Tot}\mathrm{Map}_*(\Delta^1/\partial\Delta^1,X) \simeq \mathrm{Map}_*(|\Delta^1/\partial\Delta^1|,X)\simeq LX,\]
	where $LX$ the free loop space for $X$, and
	\[\mathrm{Tot}\mathrm{Map}_*(\Delta^1/\partial\Delta^1,X)\simeq \mathrm{Tot}(\mathrm{coTHH}^{\mathcal{S}}_{\bullet}(X))=\mathrm{coTHH}^{\mathcal{S}}(X).\]
	Therefore, we have $\mathrm{coTHH}^{\mathbb{S}}(\Sigma_+^{\infty} X)\simeq \Sigma_+^{\infty} \mathrm{coTHH}^{\mathcal{S}}(X)$.
\end{exmp}

\section{The monoidality of the cotensor product}\label{sec: prop of cotensor}
In this section, we fix a connective $\mathbb{E}_{\infty}$-ring spectrum $R$ and study properties of the cotensor product in the symmetric monoidal $\infty$-category $\mathrm{Mod}_R$. For notational simplicity, we shall write $\otimes_R$ as $\otimes$. A coalgebra in $\mathrm{Mod}_R$ will be called an $R$-coalgebra. In general, since tensor products do not preserve limits separately in each variable, the cotensor product lacks many desirable properties compared to the relative tensor product. To overcome this technical challenge, we shall restrict our consideration to connective comodules over coaugmented coalgebras whose coaugmentation coideal is $2$-connective. We denote the $\infty$-category of such coalgebras by $\mathrm{CoAlg}(\mathrm{Mod}_R)_{R/}^{\geq 2}$. Note that all our conclusions remain true for bounded below comodules.

By the $\infty$-categorical Dold-Kan correspondence, the totalization of a cosimplicial object $X^{\bullet}$ is equivalent to the limit of the corresponding cofiltered object 
\[\cdots\to \mathrm{Tot}^{n+1}(X^{\bullet}) \to \mathrm{Tot}^{n}(X^{\bullet}) \to \cdots \to \mathrm{Tot}^0(X^{\bullet}).\]
Therefore, the question of whether the tensor product preserving totalizations separately in each variable can be reduced to whether the tensor product functor preserves cofiltered limits separately in each variable. The following lemma will be used repeatedly in our argument:
\begin{lem}\label{com}
	Let $X$ be a functor from $\mathrm{N}(\mathbb{Z}^{op}_{\geq0})$ to $\mathrm{Mod}_R$ such that the fiber of $X_{n+1}\to X_n$ is $n$-connective, then for any connective $R$-module $Y$ and the natural map 
\[Y\otimes \varprojlim_{i} X_i\to \varprojlim_{i} (Y\otimes X_i)\]
    is an equivalence.
\end{lem}
\begin{proof}
By the assumption that the fiber of $X_{n+i+1}\to X_{n+i}$ is $(n+i)$-connective, we see that $\pi_{n+1}(X_{n+i+1})\to \pi_{n+1}(X_{n+i})$ is an isomorphism for $i\geq 2$. Hence, 
	in the Milnor exact sequence
    \[0\to {\varprojlim_{i}}^1\pi_{n+1}(X_i)\to \pi_n(\varprojlim_{i} X_i)\to \varprojlim_{i}\pi_n(X_i)\to 0,\]
    the left term is zero and 
    \[\pi_n(\varprojlim_{i} X_i) \cong \varprojlim_{i}\pi_n(X_i) \cong \pi_n(X_j),\]
    for any $j\geq n+1$.
    Moreover, this isomorphism is induced from natural projections $p_j:\varprojlim_{i} X_i\to X_j,$ in other words, the fiber $F_j$ of $p_j$ is $(j-1)$-connective. Now, the following fiber sequence 
    \[\varprojlim_{i} (Y\otimes F_i)\to Y\otimes \varprojlim_{i} X_i\to \varprojlim_{i} (Y\otimes X_i)\]
    shows that $Y\otimes \varprojlim_{i} X_i\simeq \varprojlim_{i} (Y\otimes X_i)$, since $Y$ and $R$ are connective. 
\end{proof}

Now, for a nonempty finite set $S$, let $\mathcal{P}(S)$ be the partially ordered set of all nonempty subsets of $S$ and $\mathcal{P}^+(S)$ be the partially ordered set of all subsets of $S$. By \cite[Lemma 1.2.4.17]{HA}, the functor
\[\mathcal{P}([n])\simeq \Delta^{inj}_{/[n]}\hookrightarrow \Delta_{\leq n}\]
is right cofinal. In particular, the limit of the composite functor 
\[\mathcal{P}([n])\to \Delta_{\leq n}\xrightarrow{\Omega_R^{\bullet}(M,C,N)} \mathrm{Sp}\]
 is precisely $\mathrm{Tot}^n(\Omega_R^{\bullet}(M,C,N))$. Moreover, the inclusion $[n] \xrightarrow{i_n} [n+1]$, mapping $i \mapsto i$, induces a natural map:
\[i_n^*: \mathrm{Tot}^{n+1}(\Omega_R^{\bullet}(M,C,N)) \longrightarrow \mathrm{Tot}^n(\Omega_R^{\bullet}(M,C,N)).\] 

Our next step is to analyze the connectivity of the map $i_n^*$. Let $M$ be a right $C$-comodule, we denote the map $M\to \mathrm{Tot}^n(\Omega_R^{\bullet}(M,C,C))$ induced by the coaction map of $M$ by $f_n$.

\begin{lem}\label{connective}
Given $C\in \mathrm{CoAlg}(\mathrm{Mod}_R)_{R/}^{\geq 2}$, a connective right $C$-comodule $M$ and a connective left $C$-comodule $N$, the fiber of the natural map $\mathrm{Tot}^{n+1}(\Omega_R^{\bullet}(M,C,N)) \xrightarrow{i_n^*} \mathrm{Tot}^n(\Omega_R^{\bullet}(M,C,N))$ is $n$-connective.
\end{lem}

\begin{proof}
    First, note that for any non-empty subset $S \subseteq [n]$, the map induced by $S \hookrightarrow S \cup \{n+1\}$ is precisely $d^{|S|+1}$ which is induced by the coaction map of the left $C$-comodule $N$.
    Hence, by the dual of \cite[Lemma 1.2.4.15]{HA} applied to $K=\mathcal{P}([n+1]\setminus\{n\})$, we obtain the following pullback
    \begin{equation}\label{eq:pullback}
    \begin{tikzcd}
\mathrm{Tot}^{n+1}(\Omega_R^{\bullet}(M,C,N)) \arrow[r,"i_n^*"] \arrow[d] \arrow[dr,phantom,"\lrcorner", very near start] & \mathrm{Tot}^n(\Omega_R^{\bullet}(M,C,N)) \arrow[d] \\
M\otimes N \arrow[r,"f_n\otimes N"] & \mathrm{Tot}^n(\Omega_R^{\bullet}(M,C,C))\otimes N,
\end{tikzcd}
   	\end{equation}
   	since $K^{\triangleleft}\cong \mathcal{P}^+([n+1]\setminus\{n\})$, $K^{\triangleleft}\times \Delta^1\cong \mathcal{P}^+([n+1])$ and $\mathcal{P}([n])$ is a finite simplicial set. 
Therefore, to prove that $i_n^*$ is $(n+1)$-connective, it suffices to prove that $f_n\otimes N$ is $(n+1)$-connective. We proceed by induction on $n$ to show that $f_n$ is $(n+1)$-connective. Since both $R$ and $N$ are connective, the desired result follows. 

For the case $n=0$, the map $f_0$ coincides with $\rho$, the coaction map of $M$. Since $\epsilon \circ \rho\simeq id_M$, it follows that the composite 
    \[\pi_*(M)\xrightarrow{\rho_*}\pi_*(M\otimes C)\cong\pi_*(M)\oplus\pi_*(M\otimes \overline{C})\xrightarrow{\epsilon_*}\pi_*(M)\]
    is the identity. Note that $\epsilon_*$ is an isomorphism when restricted to $\pi_*(M)$ as $C$ is coaugmented. Moreover, because $R$ and $M$ are connective and the coaugmentation $\overline{C}$ is $2$-connective, we deduce the induced maps $\rho_i$ are isomorphisms for all $i\leq 1$. Finally, examining the long exact sequence of homotopy groups associated to the fiber sequence $F^0\to M \xrightarrow{\rho} M\otimes C$, we conclude that the fiber $F^0$ is 0-connective.
    
 We now assume that $\mathrm{fib}(f_{n-1})$ is $(n-1)$-connective and show that $\mathrm{fib}(f_n)$ is $n$-connective. For this, consider the following commutative diagram in which every row and every column is exact:
 \begin{equation*}
    \begin{tikzcd}
\Sigma^{-1}F^n \arrow[r] \arrow[d] & 0 \arrow[r] \arrow[d] & F^n \arrow[d]\\
\mathrm{fib}(f_n) \arrow[r] \arrow[d] & M \arrow[r] \arrow[d] & \mathrm{Tot}^n\Omega_R^{\bullet}(M,C,C) \arrow[d]\\
\mathrm{fib}(f_{n-1}) \arrow[r] & M \arrow[r] & \mathrm{Tot}^{n-1}\Omega_R^{\bullet}(M,C,C).
\end{tikzcd}
   	\end{equation*}
As we noted before, since $\mathrm{fib}(f_{n-1})$ is $(n-1)$-connective, so is $F^n\simeq \mathrm{fib}(f_{n-1})\otimes C$. We claim that the map $\mathrm{fib}(f_{n-1})\to F^n$ induced by the rightmost column fiber sequence in the above diagram is equivalent to 
    \[\mathrm{fib}(f_{n-1}) \xrightarrow{\rho} \mathrm{fib}(f_{n-1})\otimes C\simeq F^n.\]
    Indeed, the natural map
    \[M\to \mathrm{Tot}^i(\Omega_R^{\bullet}(M,C,C))\to M\otimes C\xrightarrow{\epsilon}M\]
    exhibits $M$ as a summand of $\mathrm{Tot}^i(\Omega_R^{\bullet}(M,C,C))$, for $i\geq 0$. Therefore, applying the following Lemma and replacing $N$ by $C$ in diagram \ref{eq:pullback} yields the above claim. Hence an argument analogous to the $n=0$ case shows that $\mathrm{fib}(f_n)$ is $n$-connective, since by hypothesis $C$ is $2$-connective and by the induction hypothesis $\mathrm{fib}(f_{n-1})$ is $(n-1)$-connective.
\end{proof}

\begin{lem}
	Given the following commutative diagram in a stable $\infty$-category $\mathcal{C}$ such that every row and every column is exact:
 \begin{equation*}
    \begin{tikzcd}
\Sigma^{-1}F \arrow[r] \arrow[d] & 0 \arrow[r] \arrow[d] & F \arrow[d]\\
\Sigma^{-1}Y \arrow[r] \arrow[d] & X \arrow[r, "i_1"] \arrow[d, "id_X"] & X\oplus Y \arrow[d, "id_X\oplus f"]\\
\Sigma^{-1}Z \arrow[r] & X \arrow[r, "i_1"] & X\oplus Z.
\end{tikzcd}
   	\end{equation*}
   	Then the map $\Sigma^{-1}Z\to F$ induced by the rightmost vertical fiber sequence in the above diagram is equivalent to the map induced on fibers by the horizontal maps in the following diagram:
   	    \begin{equation*}
    \begin{tikzcd}
X \arrow[r, "i_1"] \arrow[d, "i_1"] & X\oplus Z \arrow[d, "id"]\\
X\oplus Y \arrow[r, "id_X\oplus f"] & X\oplus Z
\end{tikzcd}
\end{equation*}
\end{lem}

\begin{proof}
	Since $X\to 0$ factors through $X\to X\oplus Y\to 0$, a diagram chase shows that it suffices to prove the leftmost vertical map in the following diagram is the identity:
	   	    \begin{equation*}
    \begin{tikzcd}
F\arrow[r] \arrow[d] & Y \arrow[r, "f"] \arrow[d] & Z \arrow[d]\\
F\arrow[r] & 0 \arrow[r] & \Sigma F,
\end{tikzcd}
\end{equation*}
This is equivalent to saying that both squares in the following diagram are pullbacks:
\begin{eqnarray*}
\begin{tikzcd}
F \arrow[r] \arrow[d] \arrow[dr,phantom,"\lrcorner", very near start] & Y \arrow[r] \arrow[d] \arrow[dr,phantom,"\lrcorner", very near start] & 0 \arrow[d] \\
0 \arrow[r] & Z \arrow[r] & \Sigma F.
\end{tikzcd}	
\end{eqnarray*}
\end{proof}

Now, given connective $R$-coalgebras $A, B$, the $\infty$-category $_A\mathrm{CoMod}_B$ is complete by \cite[Corollary]{Pe22}. However, the forget functor $U:{_A\mathrm{CoMod}_B} \to \mathrm{Mod}_R$ does not preserve limits (it preserves colimits). Nevertheless, let $C$ be a connective $R$-coalgebra, $M$ be a connective $A$-$C$-bicomodule and $N$ a connective $C$-$B$-bicomodule, using Lemma \ref{com}, we can prove that the totalization of $\Omega_R^{\bullet}(M,C,N)$ in $_A\mathrm{CoMod}_B$ can be computed in the $\infty$-category $\mathrm{Mod}_R$, or equivalently in $\mathrm{Sp}$.

\begin{prop}\label{prop:suqare in LcMod}
	Let $A, B, C, M$ and $N$ be as above, the object $M\square_C N\in \mathrm{Mod}_R$ admits an $A$-$B$-bicomodule structure and is equivalent to the totalization of $\Omega_R^{\bullet}(M,C,N)$ in $_A\mathrm{CoMod}_B$.
\end{prop}
	
\begin{proof}
    In Lemma \ref{com}, let $X = \{\mathrm{Tot}^n(\Omega_R^{\bullet}(M,C,N))\}_{n\geq 0}$. We obtain 
    \[ L \otimes (M\square_C N) \cong (L\otimes M)\square_C N,\ \forall L\in \mathrm{Mod}_R^{\mathrm{cn}},\]
    since each $\mathrm{Tot}^n(\Omega_R^{\bullet}(M,C,N))$ is equivalent to a finite limit and the fiber of the natural map 
    \[\mathrm{Tot}^{n+1}(\Omega_R^{\bullet}(M,C,N)) \to \mathrm{Tot}^n(\Omega_R^{\bullet}(M,C,N))\]
    is $n$-connective by Lemma \ref{connective}. Note that, this also implies that $(L\otimes M)\square_C N$ is connective for any connective $R$-module $L$. Now, since $A$ and $B$ are connective, we can apply \cite[Proposition 3.1.1.16]{HA} to $\mathcal{C}=(\mathrm{Mod}_R^{\mathrm{cn}})^{\mathrm{op}}$  and $X=\langle 1 \rangle \in \mathrm{N}(\mathrm{Fin}_*)$, which implies that $M\square_C N$ is an operadic colimit in the symmetric monoidal $\infty$-category $(\mathrm{Mod}_R^{\mathrm{cn}})^{\mathrm{op}}$. Hence, it follows from \cite[Proposition 4.4.2.5]{HA} that $M\square_C N$ is the geometric realization of $\Omega_R^{\bullet}(M,C,N)$ in the $\infty$-category $(\mathrm{Mod}_R^{\mathrm{cn}})^{\mathrm{op}}$. The result is obtained by passing to the opposite category.
\end{proof}

As a corollary, we have the following base change adjunction for comodules.

\begin{cor}
	Let $f:C\to D$ be a morphism of connective $R$-coalgebras with $D$ an object of $\mathrm{CoAlg}(\mathrm{Mod}_R)_{R/}^{\geq 2}$. There exists an adjoint pair
	\[f_* \colon \mathrm{CoMod}_C(\mathrm{Mod}_R) \leftrightarrows \mathrm{CoMod}_D(\mathrm{Mod}_R) \colon -\square_D C,\]
	where $f_*$ is the corestriction of scalars functor.
\end{cor}
\begin{proof}
	By Proposition \ref{prop:suqare in LcMod},  the functor $-\square_D C$ is well-defined. The remaining proof proceeds completely analogously to the algebraic case. 
\end{proof}

By naturality of cotensor product and Proposition \ref{prop:suqare in LcMod}, we obtain the following commutative diagram
\begin{equation*}
    \begin{tikzcd}
       \mathrm{BiCoMod}(\mathrm{Mod}_R^{\mathrm{cn}})\times_{\mathrm{CoAlg}(\mathrm{Mod}_R)_{R/}^{\geq 2}}\mathrm{BiCoMod}(\mathrm{Mod}_R^{\mathrm{cn}})\arrow[r,"\square"] \arrow[d, "U"] & \mathrm{BiCoMod}(\mathrm{Mod}_R) \arrow[d, "U"] \\
       \mathrm{RCoMod}_C(\mathrm{Mod}_R)\times_{\mathrm{CoAlg}(\mathrm{Mod}_R)} \mathrm{LCoMod}(\mathrm{Mod}_R) \arrow[r,"\square"] & \mathrm{Mod}_R,
\end{tikzcd}
\end{equation*}
where $U$ represents the forgetful functors to the respective categories. Both $\infty$-categories appearing in the upper horizontal part of the above diagram are symmetric monoidal by \cite[Propositon 3.2.4.3]{HA}. Hence, our next objective is to prove that the functor $\square$ is symmetric monoidal. 

\begin{lem}\label{i=1,2}
    Let $A_i, C_i$ and $B_i$ be objects of $\mathrm{CoAlg}(\mathrm{Mod}_R)_{R/}^{\geq 2}$ for $i=1,2$. Given connective $A_i$-$C_i$-bomodules $M_i$ and $C_i$-$B_i$-bimodules $N_i$ for $i=1,2$, we have an equivalence of $(A_1\otimes A_2)$-$(B_1\otimes B_2)$-bimodules
    \begin{equation}\label{eq:sym}
    	(M_1\square_{C_1} N_1)\otimes (M_2\square_{C_2} N_2)\simeq (M_1\boxtimes M_2) \square_{(C_1\otimes C_2)} (N_1\boxtimes N_2).
    \end{equation}
\end{lem}
\begin{proof}
By Lemma \ref{connective}, the fiber of 
\[i_n^*:\mathrm{Tot}^{n+1}(\Omega_R^{\bullet}(M_i,C_i,N_i)) \to \mathrm{Tot}^n(\Omega_R^{\bullet}(M_i,C_i,N_i)),\ \ i=0,1\]
 is $n$-connective. Hence, Lemma \ref{com} shows that for any connective $R$-module $N$ and the natural map 
   	 \[N\otimes (M_i\square_C N_i) \to \varprojlim_{n} (N\otimes \mathrm{Tot}^n(\Omega_R^{\bullet}(M_i,C_i,N_i)))\]
   	 is an equivalence. Equivalently, this can be formulated as an equivalence of totalizations of cosimplicial objects:
   	 \[N\otimes \varprojlim_{[k]\in\Delta}\Omega^k_R(M_i,C_i,N_i)\xrightarrow{\simeq} \varprojlim_{[k]\in\Delta}\Omega^k_R(N\otimes M_i,C_i,N_i)\]
   	 Since $\mathrm{Tot}^0(\Omega_R^{\bullet}(M_i,C_i,N_i)))=M_i \otimes N_i$ is connective, we deduce that $\mathrm{Tot}^n(\Omega_R^{\bullet}(M_i,C_i,N_i))$ and $M_i\square_{C_i} N_i$ are connective, by induction and the Milnor exact sequence. Now, the above discussion shows that 
    \begin{equation*}
    	\begin{split}
    		&(M_1\square_{C_1} N_1)\otimes (M_2\square_{C_2} N_2)\\
    		 & \simeq \varprojlim_{[k]\in\Delta} \Omega_R^k(M_1,C_1,N_1) \otimes \varprojlim_{[l]\in\Delta} \Omega_R^l(M_2,C_2,N_2)\\	
    		 & \simeq \varprojlim_{[k]\in\Delta} [\Omega_R^k(M_1,C_1,N_1) \otimes \varprojlim_{[l]\in\Delta} (\Omega_R^l(M_2,C_2,N_2)]\\ 
    		 & \simeq \varprojlim_{[k], [l]\in\Delta} \Omega_R^k(M_1,C_1,N_1) \otimes \Omega_R^l(M_2,C_2,N_2),\\
    	   \end{split}
          \end{equation*} 
    since $\Delta$ is cosifted, the above is equivalent to     	
     \begin{equation*}
    	\begin{split} 
    		& \varprojlim_{[k]\in\Delta} \Omega_R^k(M_1,C_1,N_1)) \otimes \Omega_R^k(M_2,C_2,N_2)\\
    		& \simeq \varprojlim_{[k]\in\Delta} \Omega_R^k(M_1\otimes M_2,C_1\otimes C_2,N_1\otimes N_2) \\
    		& \simeq (M_1\boxtimes M_2) \square_{(C_1\otimes C_2)} (N_1\boxtimes N_2).
    	\end{split}
    \end{equation*}
    By Proposition \ref{prop:suqare in LcMod}, this equivalence is an equivalence of $(A_1\otimes A_2)$-$(B_1\otimes B_2)$-bimodules.
\end{proof}

\begin{proof}[proof of Theorem \ref{th: sym}]
    Let $A_i,\ C_i$ and $B_i$ be $R$-coalgebras for $1\leq i\leq n$. By the definition of cotensor product, for each $n\geq 1$, there is a natural comparison map
    \[\boxtimes_{i=1}^n (M_i\square_{C_i} N_i) \to (\boxtimes_{i=1}^n M_i)\square_{(\boxtimes_{i=1}^n C_i)}(\boxtimes_{i=1}^n N_i),\]
    where $M_i$ are connective $A_i$-$C_i$-comodules and $N_i$ are connective $C_i$-$B_i$-comodules for $1\leq i \leq n$.  Moreover, we can deduce from Proposition \ref{prop:suqare in LcMod} that both sides of the above map belongs to ${_{\otimes A_i}}\mathrm{CoMod}_{\otimes B_i}$, if $C_i\in \mathrm{CoAlg}(\mathrm{Mod}_R)_{R/}^{\geq 2}$ for $1\leq i\leq n$. In other words, $\square$ is well-defined and colax symmetric monoidal. Hence, it remains to prove that this map is an equivalence. To see this, unwinding the definitions and proceeding by induction, it suffices to prove that 
	\[(M_1\square_{C_1} N_1)\boxtimes (M_2\square_{C_2} N_2)\simeq (M_1\boxtimes M_2) \square_{(C_1\otimes C_2)} (N_1\boxtimes N_2)\]
	in the $\infty$-category ${_{A_1\otimes A_2}}\mathrm{CoMod}_{B_1\otimes B_2}$, which follows from Lemma \ref{i=1,2}.
\end{proof}

Now, given maps $A\to C$ and $B\to C$ of coalgebras, we can regard $A$ as an right $C$-comodule and $B$ as a left $C$-comodule, by composing with the corresponding mappings (see \cite[Corollary 4.2.3.2]{HA}). In general, $A\square_C B$ constitutes merely an object in the $\infty$-category $\mathrm{Mod}_R$. However, the following Proposition shows that under certain conditions, $C$ admits higher comultiplicative structures.

\begin{prop}\label{pullback}
	Let $C$ be an object of $\mathrm{CoCAlg}(\mathrm{Mod}_R)_{R/}^{\geq 2}$ and $A,B\in\mathrm{CoCAlg}(\mathrm{Mod}_R^{\mathrm{cn}})$, then the pullback of the following cospan
	\begin{equation*}
    \begin{tikzcd}
        A \arrow[r] & C & B \arrow[l]
    \end{tikzcd}
    \end{equation*}
    in the $\infty$-category $\mathrm{CoCAlg}(\mathrm{Mod}_R)$ is $A\square_C B$.
\end{prop}
\begin{proof}
	First, recall that $\mathrm{CoAlg}(\mathrm{Mod}_R)$ inherits a symmetric monoidal structure, where the tensor product of two algebras is given by their tensor product in $\mathrm{Mod}_R$, by \cite[Example 3.2.4.4]{HA}. Since tensor product over $R$ preserves connectivity and tensor product of two coaugmented coalgebras is also augmented, $\mathrm{CoAlg}(\mathrm{Mod}_R^{\mathrm{cn}})$ and $\mathrm{CoAlg}(\mathrm{Mod}_R)_{R/}^{\geq 2}$ are symmetric monoidal subcategories of $\mathrm{CoAlg}(\mathrm{Mod}_R)$. Therefore, the full subcategory 
	\[\mathcal{D} \subset \mathrm{CoAlg}(\mathrm{Mod}_R)^{\Lambda^2_2}\]
	spanned by cospans $F \colon \Lambda^2_2 \to \mathrm{CoAlg}(\mathrm{Mod}_R^{\mathrm{cn}})$ with $F(0), F(1) \in \mathrm{CoAlg}(\mathrm{Mod}_R^{\mathrm{cn}})$ and $F(2) \in \mathrm{CoAlg}(\mathrm{Mod}_R)_{R/}^{\geq 2}$ inherits a symmetric monoidal structure from $\mathrm{CoAlg}(\mathrm{Mod}_R)^{\Lambda^2_2}$ with the pointwise tensor product. Moreover, Theorem \ref{th: sym} implies that the functor
	\begin{equation}\label{fun: squ}
			\square: \mathcal{D}\to \mathrm{Mod}_R
	\end{equation}
	is symmetric monoidal. Hence, we get a functor 
	\[\square: \mathrm{CoAlg}_{\mathbb{E}_\infty}\mathcal{D}\to \mathrm{CoAlg}_{\mathbb{E}_\infty}(\mathrm{Mod}_R)=\mathrm{CoCAlg}(\mathrm{Mod}_R).\]
	However, by the Dunn Additivity Theorem, we see that $\mathrm{CoAlg}_{\mathbb{E}_\infty}\mathcal{D}$ is the full subcategory of $\mathrm{CoCAlg}(\mathrm{Mod}_R)^{\Lambda^2_2}$ consisting of diagrams 
	\begin{equation*}
    \begin{tikzcd}
        A \arrow[r] & C & B \arrow[l]\in \mathrm{CoCAlg}(\mathrm{Mod}_R^{\mathrm{cn}})
    \end{tikzcd}
    \end{equation*}
    with $C$ a coaugmented cocommutative $R$-coalgebra with $2$-connective coaugmentation coideal. This shows that $A\square_C B$ is an object of $\mathrm{CoCAlg}(\mathrm{Mod}_R)$.
    
    To prove that $A\square_C B$ is the pullback in $\mathrm{CoCAlg}(\mathrm{Mod}_R)$, we must show that there is an equivalence of $\infty$-groupoids for any $D\in \mathrm{CoCAlg}(\mathrm{Mod}_R)$
       \[\alpha: \mathrm{Map}_{\mathrm{CoCAlg}(\mathrm{Mod}_R)^{\Lambda^2_2}}(\underline{D},F) \to \mathrm{Map}_{\mathrm{CoCAlg}(\mathrm{Mod}_R)}(D,A\square_C B),\]
       where $\underline{D}$ is the constant functor with value $D$.
       In fact, using an argument similar to the proof of Lemma \ref{i=1,2}, we can prove that $D \square_D D$ is an cocommutative $R$-coalgebra, and the equivalence $D \simeq D \square_D D$ in $\mathrm{Mod}_R$ is an equivalence of cocommutative $R$-coalgebras,  since the cosimplicial object $\Omega^{\bullet}_R(D,D,D)$ splits (more precisely, we can prove that $D^{\otimes^n}\square_{D^{\otimes^n}}D^{\otimes^n}\simeq (D \square_D D)^{\otimes n}$). Now, the $\alpha$ acts on objects via
           \[\underline{D} \to F \mapsto D\simeq D\square_D D\xrightarrow{\alpha(0) \square_{\alpha(2)} \alpha(1)} A\square_C B,\]
       as one can show that $\alpha(0) \square_{\alpha(2)} \alpha(1)$ is a map of coalgebras. Moreover, the identity map $\underline{D}\to \underline{D}$ induces a map $\eta_D:D\to \square(\underline{D})$, while the diagram
      \[\xymatrix{
       & A\square_C B \ar[d] \ar[r]  & B \ar[d]            \\
       & A \ar[r] & C}\]
      induces a map $\epsilon_F: \underline{\square (F)}\to F$. Note that all the above constructions are natural. Therefore, it suffices to verify that both
      \[\underline{D}\xrightarrow{\underline{\eta_D}} \underline{\square(\underline{D})}\xrightarrow{\epsilon_{\underline{D}}} \underline{D}\]
      and
      \[\square(F)\xrightarrow{\eta_{\square(F)}} \square(\underline{\square(F)})\xrightarrow{\square(\epsilon_F)}\square(F)\]
      are equivalence, and this verification is routine.
\end{proof}
 
 The following corollary is a variant of \cite[Proposition 4.5]{BGS22}.
 
\begin{cor}\label{cor:Car}
Let $C$ be an object of $\mathrm{CoCAlg}(\mathrm{Mod}_R)_{R/}^{\geq 2}$, the $\infty$-category $\mathrm{CoCAlg}(\mathrm{Mod}_R^{cn})_{/C}$ admits a Cartesian symmetric monoidal structure such that the tensor product of two objects is the cotensor product of this two objects over $C$.
\end{cor}
\begin{proof}
 	Given a cospan $A\xrightarrow{f} C \xleftarrow{g} B$ in $\mathrm{CoCAlg}(\mathrm{Mod}_R^{cn})$, $A\square_C B$ is a cocommutative $R$-coalgebra and admits a cocommutative $R$-coalgebra map to $C$, and hence becomes an object of $\mathrm{CoCAlg}(\mathrm{Mod}_R)_{/C}$, by Proposition \ref{prop:suqare in LcMod}. Furthermore, Proposition \ref{prop:suqare in LcMod} also implies that $A\square_C B$ is the pullback of the cospan in $\mathrm{CoCAlg}(\mathrm{Mod}_R)$. In particular, it is the product of $f$ and $g$ in the  $\infty$-category $\mathrm{CoCAlg}(\mathrm{Mod}_R)_{/C}$. Since $A\square_C B$ is connective by Lemma \ref{connective}, $A\square_C B$ is also the product of $f$ and $g$ in the  full subcategory $\mathrm{CoCAlg}(\mathrm{Mod}_R^{cn})_{/C}$ of $\mathrm{CoCAlg}(\mathrm{Mod}_R)_{/C}$. Therefore, $\mathrm{CoCAlg}(\mathrm{Mod}_R^{cn})_{/C}$ admits finite limits, and the result follows from \cite[Proposition 2.4.1.5]{HA}
\end{proof}
 
\begin{rem}
	In fact, we may apply the $\mathrm{CoAlg}_{\mathbb{E}_k}$ to Equation \ref{fun: squ} in place of the $\mathrm{CoAlg}_{\mathbb{E}_{\infty}}$ and get a functor
	\[\square: \mathrm{CoAlg}_{\mathbb{E}_k}\mathcal{D}\to \mathrm{CoAlg}_{\mathbb{E}_k}(\mathrm{Mod}_R).\]
	By the Dunn Additivity Theorem, this means that given a cospan 
	\begin{equation*}
    \begin{tikzcd}
        A \arrow[r] & C & B \arrow[l] \in \mathrm{CoAlg}_{\mathbb{E}_{k+1}}(\mathrm{Mod}_R^{\mathrm{cn}})
    \end{tikzcd}
    \end{equation*}
    with $C$ coaugmented whose coaugmentation coideal is $2$-connective, $A\square_C B$ is an $\mathbb{E}_k$-coalgebra. However, in general, $A\square_C B$ is not the pullback of the above cospan. 
\end{rem}

\section{Applications to coTHH}\label{sec: app}

In this section, we give some applications of the results established in Section \ref{sec: prop of cotensor} by drawing analogies with properties of topological Hochschild homology. First, we show that the following university property is satisfied by the coTHH. For the case of topological Hochschild homology, we refer to \cite[Proposition, \uppercase\expandafter{\romannumeral4}.2.2]{NS18}. Theorem \ref{thm:coalg} is a combination of the following two propositions.

\begin{prop}\label{prop:S^1}
Let	$C$ be an object of $\mathrm{CoCAlg}(\mathrm{Mod}_R)_{R/}^{\geq 2}$, then $\mathrm{coTHH}^R(C)$ is a cocommuatitive $R$-coalgebra. Moreover, the map $\mathrm{coTHH}^R(C) \to C$ is terminal among all maps from a cocommutative $R$-coalgebra equipped with an $S^1$ action to $C$ through cocommutative maps.
\end{prop}
\begin{proof}
	Note that the coaction of $C$ as an right or left $C\otimes C$-comodule is given by
	\[C\xrightarrow{\Delta}C\otimes C\xrightarrow{id_C\otimes\Delta} C\otimes C\otimes C,\]
	\[C\xrightarrow{\Delta}C\otimes C\xrightarrow{\Delta \otimes id_C} C\otimes C\otimes C,\]
	respectively, and $\Delta$ is a map of cocommutative $R$-coalgebras, since $C$ is cocommutative. Furthermore, $C\otimes C$ also belongs to $\mathrm{CoCAlg}(\mathrm{Mod}_R)_{R/}^{\geq 2}$. Hence, we can apply Theorem \ref{th1} and Proposition \ref{pullback} to get the following commutative digram in $\mathrm{CoCAlg}(\mathrm{Mod}_R)$
    \begin{eqnarray*}
       \begin{tikzcd}
        \mathrm{coTHH}^R(C) \arrow[r] \arrow[d] \arrow[dr,phantom,"\lrcorner", very near start] & C \arrow[r, "\simeq"] \arrow[d]  & C^{\Delta^1} \arrow[d] \\
           C \arrow[r] & C\otimes C \arrow[r, "\simeq"] & C^{\partial\Delta^1}.
       \end{tikzcd}	
    \end{eqnarray*}
    The equivalence of the bottom-right arrow follows from \cite[Corollary 3.2.4.7]{HA}. By the proof of Proposition \ref{pullback}, there is an equivalence of cocommutative $R$-coalgebras $C\simeq C\square_C C$. Now, the equivalence of the top-right arrow follows from the fact that $C^{\Delta1}$ is the totalization of the cosimplicial object 
    \[\mathrm{Map}_*(\Delta^1,C)\simeq \Omega^{\bullet}_R(C,C,C),\]
     since products in $\mathrm{CoCAlg}(\mathrm{Mod}_R)$ is given by tensor products. Furthermore, the composite of bottom maps is induced by $\partial\Delta^1\to *$, while the right vertical map is induced by the inclusion $i:\partial\Delta^1\to \Delta^1$. Therefore, the cotensor of $C$ with $S^1$ in the $\infty$-category of cocommuatitive $R$-coalgebras is exactly $\mathrm{coTHH}^R(C)$.
\end{proof}

The proof of the above proposition relies on $C$ being a cocommutative $R$-coalgebra. If $C$ is only an $\mathbb{E}_k$-coalgebra (for $k<\infty$), the theorem no longer holds. Nevertheless, in this case, we can still obtain results analogous to topological Hochschild homology.

\begin{prop}\label{Ek}
	The functor $\mathrm{coTHH}^R$ can be lifted to a functor from the $\infty$-category of coaugmented $\mathbb{E}_k$-R-coalgebras with $2$-connective coaugmentation coideals to the $\infty$-category of $\mathbb{E}_{k-1}$-R-coalgebras:
	\[\mathrm{coTHH}^R: \mathrm{CoAlg}_{\mathbb{E}_k}(\mathrm{Mod}_R)_{R/}^{\geq 2}\to \mathrm{CoAlg}_{\mathbb{E}_{k-1}}(\mathrm{Mod}_R),\]
	for $k\geq 1$.
\end{prop}

\begin{proof}
	Note that $C_1^e\otimes C_2^e\simeq (C_1\otimes C_2)^e$ as coalgebras, hence by Lemma \ref{i=1,2} we obtain 
	\[(C_1\square_{C_1^e} C_1)\otimes (C_2\square_{C_2^e} C_2)\simeq (C_1\boxtimes_R C_2) \square_{(C_1\otimes C_2)^e} (C_1\boxtimes_R C_2).\]
	Therefore, Theorem 1 shows that $\mathrm{coTHH}^R$ is symmetric monoidal as a functor from $\mathrm{CoAlg}(\mathrm{Mod}_R)_{R/}^{\geq 2}$ to $\mathrm{Mod}_R.$ Now, the result follows from the Dunn Additivity Theorem.
\end{proof}

Next, we consider the associativity of $\square$. For this purpose, let $C$ and $D$ be two coaugmented coalgebras with $2$-connective coaugmentation coideals,  $L$ be an right $C$-comodule, $M$ be a $C$-$D$-bicomodule, and $N$ be a left $D$-comodule, with the latter three comodules being connective. 

\begin{lem}\label{ass}
With the above notation, we have
	\[(L\square_C M)\square_D N\simeq L\square_C(M\square_D N).\]
\end{lem}

\begin{proof}
	By Proposition \ref{prop:suqare in LcMod}, $L \square_C M$ is a right $D$-comodule and $M \square_D N$ is a left $C$-comodule, so the notation makes sense. An analogous proof to that of Lemma \ref{i=1,2} establishes the result.
\end{proof}

As an application of this lemma, we use it to examine the Morita-Takeuchi invariance of $\mathrm{coTHH}$. We use the following lemma to define Morita-Takeuchi equivalent coalgebras, where again the connectivity assumption is necessary:
\begin{lem}\label{lem:morita}
	Given $C,D\in \mathrm{CoAlg}(\mathrm{Mod}_R)_{R/}^{\geq 2}$, the following two conditions are equivalent:
	\begin{enumerate}
		\item[$\mathrm{(i)}$] There is an $R$-linear equivalence of $\infty$-categories
		    \[F: \mathrm{CoMod}_C(\mathrm{Mod}_R^{>\infty}) \xrightarrow{\sim} \mathrm{CoMod}_D(\mathrm{Mod}_R^{>\infty}).\]
		\item[$\mathrm{(ii)}$] There exist a bounded below $C$-$D$-bicomodule $P$ and a bounded below $D$-$C$-bicomodule $Q$ such that $P\square_D Q \simeq C$ and $Q\square_C P\simeq D$ as $C$-bicomodules and $D$-bicomodules, respectively.
	\end{enumerate}
	Moreover, if these two conditions are satisfied, then $F$ is equivalent to $-\square_C P$ and its inverse is equivalent to $-\square_D Q$. Here $\mathrm{Mod}_R^{>\infty}$ is the $\infty$-category of bounded below $R$-modules.
\end{lem}
\begin{proof}
	Clearly, $\mathrm{(ii)}\Rightarrow\mathrm{(i)}$ follows directly from Proposition \ref{prop:suqare in LcMod}, Lemma \ref{ass} and  Remark \ref{bounded}. It remains to prove that $\mathrm{(i)}\Rightarrow\mathrm{(ii)}$, and the remaining statement follows immediately. In fact, since any bounded below right $C$-comodule $M$ is equivalent to the totalization of $\Omega^{\bullet}_R(M,C,C)$ in $\mathrm{CoMod}_C(\mathrm{Mod}_R)$, $F(M)$ is equivalent to the totalization of 
	\[F(\Omega^{\bullet}_R(M,C,C))\simeq \Omega^{\bullet}_R(M,C,F(C))\]
	in $\mathrm{CoMod}_D(\mathrm{Mod}_R)$, where $F(C)$ is a left $C$-comodule as in \cite[Remark, 4.6.2.9]{HA}. By assumption, $F(C)$ is bounded below, hence Proposition \ref{prop:suqare in LcMod} implies that the totalization of $\Omega^{\bullet}_R(M,C,F(C))$ can be computed in $\mathrm{Mod}_R$. That is
	 \[F(M)\simeq M\square_C F(C).\]
	 Similarly, let $G$ be the $R$-linear inverse of $F$, than $G(N)\simeq N\square_D G(D)$ for any bounded below right $D$-comodule $N$. By Lemma \ref{ass}, taking $P=F(C)$ and $Q=G(D)$ yields the required condition.
\end{proof}

Two coalgebras $C$ and $D$ in $\mathrm{CoAlg}(\mathrm{Mod}_R)_{R/}^{\geq 2}$ are Morita-Takeuchi equivalent if they satisfy the conditions in the last Lemma. We use the dual of Dennis-Waldhausen Morita argument, see \cite[Proposition]{BM12}, to prove that coTHH sends Morita-Takeuchi equivalent coalgebras in $\mathrm{CoAlg}(\mathrm{Mod}_R)_{R/}^{\geq 2}$ to equivalence. The case $R=H\mathbb{K}$, where $\mathbb{K}$ is an ordinal commutative ring with global dimension zero, was proved in \cite{KP25}, and our proof follows essentially the same approach.

\begin{thm}\label{Morita}
	Suppose that $C,D\in \mathrm{CoAlg}(\mathrm{Mod}_R)_{R/}^{\geq 2}$ are Morita-Takeuchi equivalent coalgebras, than there is an equivalence
	\[\mathrm{coTHH}^R(C)\simeq \mathrm{coTHH}^R(D).\]
\end{thm}
\begin{proof}
Choose bounded below bicomodules $P$ and $Q$ in the last Lemma, we may assume that they are connective by Remark \ref{bounded}. Recall that $P\square_D Q$ can be identified with the totalization of the functor
\[\Omega^{\bullet}_R(P,D,Q):\Delta\to \mathrm{Env}(\mathbf{Assoc}_{\mathrm{RL}})^{\mathrm{op}}\to \mathrm{Env}(\mathrm{Mod}_R^{\mathrm{op}}) ^{\mathrm{op}}\xrightarrow{\otimes} \mathrm{Mod}_R,\]
see \cite[\S 3.2]{AF15} for example. Since $P$ and $Q$ are left and right $C$-comodules respectively, we see that $\Omega^{\bullet}_R(P,D,Q)$ is a $\underline{C}$-bicomodule in the $\infty$-category $\mathrm{Fun}(\Delta,\mathrm{Mod}_R)$ with the pointwise tensor product, where the coalgebra $\underline{C}$ is the constant functor with value $C$. Hence, taking the coTHH of $\underline{C}$ with coefficient $\Omega^{\bullet}_R(P,C,Q)$ in the $\infty$-category $\mathrm{Fun}(\Delta,\mathrm{Mod}_R)$, we obtain a functor:
\[\mathrm{coTHH}^{\bullet}(\Omega^{\bullet}_R(P,D,Q);\underline{C}):\Delta \times \Delta \to \mathrm{Mod}_R.\]
Composing it with the diagonal map, we denote the resulting functor by $F$. Unwinding the definitions and using the equivalence $\mathrm{Env}(\mathcal{C})=\mathcal{C}^{\otimes}_{\mathrm{act}} $, we see that $F$ factors through $\Delta\to ((\mathrm{Mod}_R^{\mathrm{op}})_{\mathrm{act,even}}^{\otimes})^{\mathrm{op}}\to \mathrm{Mod}_R$, where $(\mathrm{Mod}_R^{\mathrm{op}})_{\mathrm{act,even}}^{\otimes}$ is defined to be the pullback of the following diagram
\begin{equation*}
	    \begin{tikzcd}
(\mathrm{Mod}_R^{\mathrm{op}})_{\mathrm{act,even}}^{\otimes} \arrow[r] \arrow[d, "p"] \arrow[dr,phantom,"\lrcorner", very near start] & (\mathrm{Mod}_R^{\mathrm{op}})_{\mathrm{act}}^{\otimes} \arrow[d] \\
\mathrm{Act}(\mathrm{N}(\mathrm{Fin}_*)_{\mathrm{even}}) \arrow[r] & \mathrm{Act}(\mathrm{N}(\mathrm{Fin}_*))
\end{tikzcd}
\end{equation*}
where $\mathrm{N}(\mathrm{Fin}_*)_{\mathrm{even}}$ is the full subcategory of $\mathrm{N}(\mathrm{Fin}_*)$ generated by the objects $\langle 2n \rangle$ for $n\geq 0$. Now, since $p$ is coCartesian, so is $p^{\Delta}: \mathrm{Fun}(\Delta,(\mathrm{Mod}_R^{\mathrm{op}})_{\mathrm{act,even}}^{\otimes})\to \mathrm{Fun}(\Delta,\mathrm{Act}(\mathrm{N}(\mathrm{Fin}_*)_{\mathrm{even}}))$. Therefore, the sequence of maps $\tau_n:\langle 2n \rangle \to \langle 2n \rangle$
    \begin{equation*}
          \begin{split}
          \tau_n(i)= \left \{
          \begin{array}{ll}
           i+n,     & \mathrm{if}\ 1\leq i\leq n,\\
           i-n,     & \mathrm{if}\ n+1\leq i\leq 2n,
          \end{array}
          \right.
          \end{split}
    \end{equation*}
determines another functor $F^{\prime} :\Delta\to ((\mathrm{Mod}_R^{\mathrm{op}})_{\mathrm{act,even}}^{\otimes})^{\mathrm{op}}\to \mathrm{Mod}_R$ which is equivalent to $F$ by \cite[Proposition 2.4.1.5]{HTT}. It follows that 
\[\varprojlim_{[i]\in \Delta} P\otimes D^{\otimes i}\otimes Q\otimes C^{\otimes i} \simeq \varprojlim_{[i]\in \Delta} F \simeq \varprojlim_{[i]\in \Delta} F^{\prime}.\]
A similarly argument for $\mathrm{coTHH}^{\bullet}(\Omega^{\bullet}_R(Q,C,P);\underline{D})$ shows that
\[\varprojlim_{[j]\in \Delta} Q\otimes C^{\otimes j}\otimes P\otimes D^{\otimes j}\simeq \varprojlim_{[j]\in \Delta} F^{\prime\prime},\]
where $F^{\prime\prime}$ can be identified with the opposite of the cosimplicial object $F^{\prime}$.
Hence, 
\[\varprojlim_{[i]\in \Delta} P\otimes D^{\otimes i}\otimes Q\otimes C^{\otimes i} \simeq \varprojlim_{[j]\in \Delta} Q\otimes C^{\otimes j}\otimes P\otimes D^{\otimes j}.\]
Now, combining Lemma \ref{com} and \ref{connective} with the fact that $\Delta$ is cosifted, and using the above equivalence, a direct computation yields
\[\mathrm{coTHH}^R(P\square_DQ;C) \simeq \mathrm{coTHH}^R(Q\square_C P;D).\]
Finally, by Lemma \ref{lem:morita}, we have $P\square_DQ\simeq C$ and $D\simeq Q\square_C P$, which implies the desired result.
\end{proof}

As an other application of the associativity of cotensor product, we establish the following Proposition for coTHH with coefficients, which is dual to results in \cite[\S\,2]{AHL10}. These results have played an important role in computations of THH (see, e.g., \cite{HW22,Lee22}), and we expect this dual result to also be useful for coTHH computations.
 
\begin{prop}
\begin{enumerate}
	\item Let $C\in \mathrm{CoAlg}_{\mathbb{E}_2}(\mathrm{Mod}_R)_{R/}^{\geq 2}$ and $M$ be a connective right $C$-comodule. Then $M$ admits a $C$-bicomodule structure, and we have the following equivalence:
   	\[\mathrm{coTHH}^R(M;C)\simeq M\square_C \mathrm{coTHH}^R(C).\]
   	\item Given a map $D\to C$ of coaugmented coalgebras with $2$-connective coaugmentation coideals and $M$ a connective $C$-$D$-bicomodule. We have the following equivalence:
   	\[\mathrm{coTHH}^R(M;C)\simeq M\square_{D\otimes C^{\mathrm{op}}} D.\]
\end{enumerate}
\end{prop}
\begin{proof}
	Since $C$ be an $\mathbb{E}_2$-coalgebra, we obtain a morphism of coalgebras $C\to C\otimes C^{op}$ which endows $M$ with a $C$-bicomodule structure. Moreover, we can deduce from Proposition \ref{Ek} that $C\square_{C\otimes C^{\mathrm{op}}} C$ is a left $C$-comodule and this structure is exactly the one appearing in the proof of Lemma \ref{ass}. Hence, there is a chain of weak equivalences:
	\[M\square_{C\otimes C^{\mathrm{op}}} C\simeq (M\square_C C)\square_{C\otimes C^{\mathrm{op}}} C\simeq M\square_C (C\square_{C\otimes C^{\mathrm{op}}} C) .\]
	Now, the first half of the proposition is a consequence of Theorem \ref{th1}. 
	
	Similarly, considering $D$ as a right $C$-comodule, we have
	\[M\square_{C\otimes C^{\mathrm{op}}} C\simeq [M\square_{D\otimes C^{\mathrm{op}}} (D\otimes C^{\mathrm{op}})]\square_{C\otimes C^{\mathrm{op}}} C.\]
	Since the cosimplicial object $\Omega_R^{\bullet}(D\otimes C^{\mathrm{op}},C\otimes C^{\mathrm{op}}, C)$ splits with the coaugmentation given by $D\to D\otimes (C^{\mathrm{op}}\otimes C)$, Proposition \ref{prop:suqare in LcMod} and Lemma \ref{ass} shows that the right-hand side is equivalent to
	\[M\square_{D\otimes C^{\mathrm{op}}} [(D\otimes C^{\mathrm{op}})\square_{C\otimes C^{\mathrm{op}}} C]\simeq M\square_{D\otimes C^{\mathrm{op}}} D.\]
	 Now, the result follows from Theorem \ref{th1}.
\end{proof}

Finally, we use Proposition \ref{prop:suqare in LcMod} to give a generalization of \cite[Proposition 5.6]{BGS22} which states that for a simply connected space $X$, $\mathrm{coTHH}^{\mathbb{S}}(\Sigma^{\infty}_+X)$ is a Hopf algebra in the homotopy category of $\mathrm{CoAlg}(\mathrm{Sp})_{/\Sigma^{\infty}_+X}$. This can be regarded as the dual of \cite[Corollary IX.3.6]{EKMM}. Let $C$ be an object of $\mathrm{CoCAlg}(\mathrm{Mod}_R)_{R/}^{\geq 2}$, by Corollary \ref{cor:Car}, the $\infty$-category $\mathrm{CoCAlg}(\mathrm{Mod}_R^{cn})_{/C}$ admits a Cartesian symmetric monoidal structure. In particular, every object is a cocommutative coalgebra object of $\mathrm{CoCAlg}(\mathrm{Mod}_R^{cn})_{/C}$ by \cite[Corollary 2.4.3.10]{HA}.

\begin{prop}\label{prop:hopf}
	Let $C$ be an object of $\mathrm{CoCAlg}(\mathrm{Mod}_R)_{R/}^{\geq 2}$. The coalgebra $\mathrm{coTHH}^R(C)$ is a commutative Hopf algebra in the $\infty$-category $\mathrm{CoCAlg}(\mathrm{Mod}_R^{cn})_{/C}$ equipped with Cartesian symmetric monoidal structure of Corollary \ref{cor:Car}.
\end{prop}
\begin{proof}
	Let $\mathcal{F}$ be the full subcategory of $\mathcal{S}_*$ spanned by the $n$-flood wedge sum of $S^1$ for $n\geq 0$. This $\infty$-category admits finite coproducts, hence $\mathcal{F}^{op}$ admits a Cartesian symmetric monoidal structure. Consider the functor
	\[C^-: \mathcal{F}^{op}\to \mathrm{CoCAlg}(\mathrm{Mod}_R), \ \ X\mapsto C^X.\]
	Since, as a functor, cotensor of $C$ with spaces preserves products, the functor $C^-$ is completely determined by $C^*$ and $C^{S^1}$. Clearly, $C^*=C$, while $C^{S^1}=\mathrm{coTHH}^R(C)$ by Proposition \ref{prop:S^1}. Furthermore, the canonical map of cocommutative $R$-coalgebras $\mathrm{coTHH}^R(C)\to C$ is induced by the pointed map $*\to S^1$. By lemma \ref{connective}, $\mathrm{coTHH}^R(C)$ is connective. Hence Proposition \ref{pullback} implies that there is an equivalence of cocommutative coalgebras
	\[C^{S^1\vee S^1}\simeq \mathrm{coTHH}^R(C)\square_C \mathrm{coTHH}^R(C)\]
	which is also connective by Lemma \ref{connective}. Using Lemma \ref{connective} and Propsition \ref{pullback} repeatedly, we obtain an equivalence of connective cocommutative coalgebras over $C$
	\[\underbrace{\mathrm{coTHH}^R(C)\square_C \cdots \square_C \mathrm{coTHH}^R(C)}_{n}\simeq \underbrace{C^{S^1}\square_C \cdots \square_C C^{S^1}}_{n}\simeq C^{\bigvee_{i=1}^n S^1}\]
	for any $n\geq 0$. Consequently, the functor $C^-$ factors through $\mathrm{CoCAlg}(\mathrm{Mod}_R^{cn})_{/C}$. By Corollary \ref{cor:Car} and the argument above, the resulting functor $\mathcal{F}^{op} \to \mathrm{CoCAlg}(\mathrm{Mod}_R^{cn})_{/C}$ is symmetric monoidal when both are equipped with the Cartesian symmetric monoidal structure. Now, the result follows from the fact that $S^1$ is a commutative Hopf algebra in the $\infty$-category $\mathcal{F}^{op}$.
\end{proof}

\begin{rem}
	When $C=\Sigma^{\infty}_+X$ with $X$ simply connected. The commutative Hopf algebra structure of $\mathrm{coTHH}^{\mathbb{S}}(\Sigma^{\infty}_+X)$ in Proposition \ref{prop:hopf} coincides with that in \cite[Proposition 5.6]{BGS22}. Furthermore, it follows from Corollary \ref{cor:Car} that $\mathrm{coTHH}(\Sigma^{\infty}_+X)\square_{\Sigma^{\infty}_+X}\mathrm{coTHH}(\Sigma^{\infty}_+X)$ in \cite[Proposition 5.6]{BGS22} can be identified with the cotensor product of two copies of $\mathrm{coTHH}^{\mathbb{S}}(\Sigma^{\infty}_+X)$ over $\Sigma^{\infty}_+X$.
\end{rem}

As an application, we have the following coB\"okstedt spectral sequence.
\begin{cor}
	Let $k$ be a field and $C$ be a cocommutative $k$-coalgebra. There is a coB\"okstedt spectral sequence converging to $\pi_{t-s}(\mathrm{coTHH}^k(C))$ with $E_2$-page
	\[E_2^{s,t}(C)=\mathrm{coHH}^k_{s+t}(C_*),\]
	which admits a comultiplication $E_r^{s,t}\to E_r^{s,t}\square_{C_*}E_r^{s,t}.$
\end{cor}
\begin{proof}
	This first statement is a special case of the Bousfield-Kan spectral sequence, and the second is a special case of \cite[Theorems 6.10]{BGS22}, we give a simpler proof by using Proposition \ref{prop:hopf}. When $C \in \mathrm{CoCAlg}(\mathrm{Mod}_k)_{k/}^{\geq 2}$, Proposition \ref{prop:hopf} implies that there is a cone in the category of graded $k$-modules
	\[\pi_*(\mathrm{coTHH}^k(C)\square_C\mathrm{coTHH}^k(C))\to \Omega_k^{\bullet}(\pi_*(\mathrm{coTHH}^k(C)),{C_*},\pi_*(\mathrm{coTHH}^k(C)))\]
	which preserves Bousfield-Kan filtration. Therefore, it induces a map of spectral sequences
	\begin{equation*}
		^{\prime}E_r^{s,t}\to E_r^{s,t}\square_{C_*}E_r^{s,t},
	\end{equation*}
	since every module over $k$ is flat. Now, the map $\mathrm{coTHH}^k(C)\to \mathrm{coTHH}^k(C)\square_C\mathrm{coTHH}^k(C)$ induces the comultiplication
	\[\Delta:E_r^{s,t}\to {^{\prime}E_r^{s,t}}\to E_r^{s,t}\square_{C_*}E_r^{s,t}.\]
\end{proof}

\begin{rem}
	Furthermore, if $C$ is an object of $\mathrm{CoCAlg}(\mathrm{Mod}_k)_{k/}^{\geq 2}$ and  $E_r^{s,t}(C)$ is coflat over $C_*$ for any $r\geq 2$, then this spectral sequence admits a $\square_{C_*}$-Hopf algebra structure by \cite[Theorem 6.14]{BGS22}.
\end{rem}

\bibliographystyle{plain}
\bibliography{ref}

@book{EKMM,
 author = {Elmendorf, A. D. and K{\v{r}}{\'{\i}}{\v{z}}, Igor and Mandell, Michael A. and May, J. P.},
 title = {Rings, modules, and algebras in stable homotopy theory. {With} an appendix by {M}. {Cole}},
 fseries = {Mathematical Surveys and Monographs},
 series = {Math. Surv. Monogr.},
 issn = {0076-5376},
 volume = {47},
 isbn = {0-8218-0638-6},
 year = {1997},
 publisher = {Providence, RI: American Mathematical Society},
 language = {English},
 zbMATH = {949383},
 Zbl = {0894.55001}
}

@book{HTT,
  title={Higher topos theory},
  author={Lurie, Jacob},
  year={2009},
  publisher={Princeton University Press}
}

@book{HA,
 author = {Lurie, Jacob},
 title = {Higher Algebra},
 year = {2017},
 note = {Preprint available at \url{www.math.ias.edu/~lurie/}},
}

@Article{HS21,
 Author = {Hess, Kathryn and Shipley, Brooke},
 Title = {Invariance properties of {coHochschild} homology},
 FJournal = {Journal of Pure and Applied Algebra},
 Journal = {J. Pure Appl. Algebra},
 ISSN = {0022-4049},
 Volume = {225},
 Number = {2},
 Pages = {26},
 Note = {Id/No 106505},
 Year = {2021},
 Language = {English},
 DOI = {10.1016/j.jpaa.2020.106505},
 zbMATH = {7241711},
 Zbl = {1454.13018}
}

@article{HS16,
 author = {Hess, Kathryn and Shipley, Brooke},
 title = {Waldhausen {{\(K\)}}-theory of spaces via comodules},
 fjournal = {Advances in Mathematics},
 journal = {Adv. Math.},
 issn = {0001-8708},
 volume = {290},
 pages = {1079--1137},
 year = {2016},
 language = {English},
 doi = {10.1016/j.aim.2015.12.019},
 zbMATH = {6538742},
 Zbl = {1331.19003}
}

@Article{AHL10,
 Author = {Angeltveit, Vigleik and Hill, Michael A. and Lawson, Tyler},
 Title = {Topological {Hochschild} homology of {{\(\ell \)}} and {{\(ko\)}}},
 FJournal = {American Journal of Mathematics},
 Journal = {Am. J. Math.},
 ISSN = {0002-9327},
 Volume = {132},
 Number = {2},
 Pages = {297--330},
 Year = {2010},
 Language = {English},
 DOI = {10.1353/ajm.0.0105},
 zbMATH = {5706732},
 Zbl = {1271.55009}
}

@Article{BGN14,
 Author = {Barwick, Clark and Glasman, Saul and Nardin, Denis},
 Title = {Dualizing cartesian and cocartesian fibrations},
 FJournal = {Theory and Applications of Categories},
 Journal = {Theory Appl. Categ.},
 ISSN = {1201-561X},
 Volume = {33},
 Pages = {67--94},
 Year = {2018},
 Language = {English},
 URL = {www.tac.mta.ca/tac/volumes/33/4/33-04abs.html},
 zbMATH = {6881671},
 Zbl = {1423.18025}
}

@Article{BP23,
 Author = {Bayindir, Haldun {\"O}zg{\"o}r and P{\'e}roux, Maximilien},
 Title = {Spanier-{Whitehead} duality for topological {coHochschild} homology},
 FJournal = {Journal of the London Mathematical Society. Second Series},
 Journal = {J. Lond. Math. Soc., II. Ser.},
 ISSN = {0024-6107},
 Volume = {107},
 Number = {5},
 Pages = {1780--1822},
 Year = {2023},
 Language = {English},
 DOI = {10.1112/jlms.12725},
 zbMATH = {7731088},
 Zbl = {1550.16031}
}

@Article{KP25,
 Author = {Klanderman, Sarah and P{\'e}roux, Maximilien},
 Title = {Trace methods for {coHochschild} homology},
 FJournal = {Mathematische Zeitschrift},
 Journal = {Math. Z.},
 ISSN = {0025-5874},
 Volume = {310},
 Number = {1},
 Pages = {51},
 Note = {Id/No 1},
 Year = {2025},
 Language = {English},
 DOI = {10.1007/s00209-025-03703-z},
 zbMATH = {8012067}
}

@article{NS18,
 author = {Nikolaus, Thomas and Scholze, Peter},
 title = {On topological cyclic homology},
 fjournal = {Acta Mathematica},
 journal = {Acta Math.},
 issn = {0001-5962},
 volume = {221},
 number = {2},
 pages = {203--409},
 year = {2018},
 language = {English},
 doi = {10.4310/ACTA.2018.v221.n2.a1},
 zbMATH = {7009201},
 Zbl = {1457.19007}
}

@article{BM12,
 author = {Blumberg, Andrew J. and Mandell, Michael A.},
 title = {Localization theorems in topological {Hochschild} homology and topological cyclic homology},
 fjournal = {Geometry \& Topology},
 journal = {Geom. Topol.},
 issn = {1465-3060},
 volume = {16},
 number = {2},
 pages = {1053--1120},
 year = {2012},
 language = {English},
 doi = {10.2140/gt.2012.16.1053},
 zbMATH = {6068621},
 Zbl = {1282.19004}
}

@article{Pe22,
 author = {P{\'e}roux, Maximilien},
 title = {The coalgebraic enrichment of algebras in higher categories},
 fjournal = {Journal of Pure and Applied Algebra},
 journal = {J. Pure Appl. Algebra},
 issn = {0022-4049},
 volume = {226},
 number = {3},
 pages = {11},
 note = {Id/No 106849},
 year = {2022},
 language = {English},
 doi = {10.1016/j.jpaa.2021.106849},
 zbMATH = {7396430},
 Zbl = {1484.16040}
}

@article{Doi81,
 author = {Doi, Yukio},
 title = {Homological coalgebra},
 fjournal = {Journal of the Mathematical Society of Japan},
 journal = {J. Math. Soc. Japan},
 issn = {0025-5645},
 volume = {33},
 pages = {31--50},
 year = {1981},
 language = {English},
 doi = {10.2969/jmsj/03310031},
 zbMATH = {3719314},
 Zbl = {0459.16007}
}

@article{Kla22,
 author = {Klanderman, Sarah},
 title = {Computations of relative topological {coHochschild} homology},
 fjournal = {Journal of Homotopy and Related Structures},
 journal = {J. Homotopy Relat. Struct.},
 issn = {2193-8407},
 volume = {17},
 number = {3},
 pages = {393--417},
 year = {2022},
 language = {English},
 doi = {10.1007/s40062-022-00312-z},
 zbMATH = {7576632},
 Zbl = {1514.16029}
}

@article{BGS22,
 author = {Bohmann, Anna Marie and Gerhardt, Teena and Shipley, Brooke},
 title = {Topological {coHochschild} homology and the homology of free loop spaces},
 fjournal = {Mathematische Zeitschrift},
 journal = {Math. Z.},
 issn = {0025-5874},
 volume = {301},
 number = {1},
 pages = {411--454},
 year = {2022},
 language = {English},
 doi = {10.1007/s00209-021-02879-4},
 zbMATH = {7507822},
 Zbl = {1493.55010}
}

@article{Mal17,
 author = {Malkiewich, Cary},
 title = {Cyclotomic structure in the topological {Hochschild} homology of {{\(DX\)}}},
 fjournal = {Algebraic \& Geometric Topology},
 journal = {Algebr. Geom. Topol.},
 issn = {1472-2747},
 volume = {17},
 number = {4},
 pages = {2307--2356},
 year = {2017},
 language = {English},
 doi = {10.2140/agt.2017.17.2307},
 zbMATH = {6762692},
 Zbl = {1390.19005}
}

@misc{CCRY22,
 author = {Carmeli, Shachar and Cnossen, Bastiaan and Ramzi, Maxime and Yanovski, Lior},
 title = {Characters and transfer maps via categorified traces},
 year = {2022},
 howpublished = {Preprint, {arXiv}:2210.17364 [math.{AT}] (2022)},
 url = {https://arxiv.org/abs/2210.17364},
 arXiv = {arXiv:2210.17364}
}

@article{AF15,
 author = {Ayala, David and Francis, John},
 title = {Factorization homology of topological manifolds},
 fjournal = {Journal of Topology},
 journal = {J. Topol.},
 issn = {1753-8416},
 volume = {8},
 number = {4},
 pages = {1045--1084},
 year = {2015},
 language = {English},
 doi = {10.1112/jtopol/jtv028},
 zbMATH = {6525834},
 Zbl = {1350.55009}
}

@article{PS18,
 author = {P{\'e}roux, Maximilien and Shipley, Brooke},
 title = {Coalgebras in symmetric monoidal categories of spectra},
 fjournal = {Homology, Homotopy and Applications},
 journal = {Homology Homotopy Appl.},
 issn = {1532-0073},
 volume = {21},
 number = {1},
 pages = {1--18},
 year = {2018},
 language = {English},
 doi = {10.4310/HHA.2019.v21.n1.a1},
 zbMATH = {6986482},
 Zbl = {1433.16040}
}

@article{Go85,
 author = {Goodwillie, Thomas G.},
 title = {Cyclic homology, derivations, and the free loopspace},
 fjournal = {Topology},
 journal = {Topology},
 issn = {0040-9383},
 volume = {24},
 pages = {187--215},
 year = {1985},
 language = {English},
 doi = {10.1016/0040-9383(85)90055-2},
 zbMATH = {3908625},
 Zbl = {0569.16021}
}

@misc{Lee22,
 author = {Lee, David Jongwon},
 title = {Integral topological {Hochschild} homology of connective complex {K}-theory},
 year = {2022},
 howpublished = {Preprint, {arXiv}:2206.02411 [math.{AT}] (2022)},
 url = {https://arxiv.org/abs/2206.02411},
 arXiv = {arXiv:2206.02411}
}

@article{HW22,
 author = {Hahn, Jeremy and Wilson, Dylan},
 title = {Redshift and multiplication for truncated {Brown}-{Peterson} spectra},
 fjournal = {Annals of Mathematics. Second Series},
 journal = {Ann. Math. (2)},
 issn = {0003-486X},
 volume = {196},
 number = {3},
 pages = {1277--1351},
 year = {2022},
 language = {English},
 doi = {10.4007/annals.2022.196.3.6},
 zbMATH = {7611907},
 Zbl = {1541.55010}
}

@article{HSS17,
 author = {Hoyois, Marc and Scherotzke, Sarah and Sibilla, Nicol{\`o}},
 title = {Higher traces, noncommutative motives, and the categorified {Chern} character},
 fjournal = {Advances in Mathematics},
 journal = {Adv. Math.},
 issn = {0001-8708},
 volume = {309},
 pages = {97--154},
 year = {2017},
 language = {English},
 doi = {10.1016/j.aim.2017.01.008},
 url = {orbilu.uni.lu/handle/10993/51963},
 zbMATH = {6686567},
 Zbl = {1361.14014}
}
	
\end{document}